\documentclass[12pt]{amsart}
\usepackage{amsfonts,amssymb,amscd,amsmath,enumerate,verbatim}
\usepackage[latin1]{inputenc}
\usepackage{latexsym}
\usepackage{pstcol,pst-plot,pst-3d}
\usepackage{mathptmx}
\usepackage{multicol}
\usepackage{amstext,amsthm}
\usepackage{graphicx}
\usepackage{enumerate}
\usepackage{epsfig}
\usepackage{grffile}
\usepackage{float}
\usepackage{hyphenat}

\psset{unit=0.7cm,linewidth=0.8pt,arrowsize=2.5pt 4}
\newpsstyle{fatline}{linewidth=1.5pt}
\newpsstyle{fyp}{fillstyle=solid,fillcolor=verylight}
\definecolor{verylight}{gray}{0.97}
\definecolor{light}{gray}{0.9}
\definecolor{medium}{gray}{0.85}

\def\frk{\mathfrak}               % font for "Fraktur"

\def\Phi{{\frk N}}
\def\opn#1#2{\def#1{\operatorname{#2}}} % to make operators
\opn\chara{char} \opn\length{\ell} \opn\pd{pd} \opn\rk{rk}
\opn\projdim{proj\,dim} \opn\injdim{inj\,dim} \opn\rank{rank}
\opn\depth{depth} \opn\grade{grade} \opn\height{height} \opn\bheight{bigheight}
\opn\embdim{emb\,dim} \opn\codim{codim}

\opn\Tr{Tr} \opn\bigrank{big\,rank}
\opn\superheight{superheight}\opn\lcm{lcm}
\opn\trdeg{tr\,deg}%\emph{
\opn\reg{reg} \opn\lreg{lreg} \opn\ini{in} \opn\lpd{lpd}
\opn\size{size}\opn{\mult}{mult}
\opn\div{div} \opn\Div{Div} \opn\cl{cl} \opn\Cl{Cl}
\opn\Spec{Spec} \opn\Supp{Supp} \opn\supp{supp} \opn\Sing{Sing}
\opn\Ass{Ass} \opn\Min{Min}
\opn\Ann{Ann} \opn\Rad{Rad} \opn\Soc{Soc}
\opn\Syz{Syz} \opn\Im{Im} \opn\Ker{Ker} \opn\Coker{Coker}
\opn\Am{Am} \opn\Hom{Hom} \opn\Tor{Tor} \opn\Ext{Ext}
\opn\End{End} \opn\Aut{Aut} \opn\id{id} \opn\ini{in}
 \opn\comp{comp}
\opn\nat{nat}
\opn\pff{pf}%   \pf exists already
\opn\Pf{Pf} \opn\GL{GL} \opn\SL{SL} \opn\mod{mod} \opn\ord{ord}
\opn\Gin{Gin} \opn\reg{reg}
\opn\Hilb{Hilb}\opn\adeg{adeg}\opn\std{std}\opn\ip{infpt}
\opn\Pol{Pol}
\opn\sat{sat}
\opn\Var{Var}
\opn\Gen{Gen}
\opn\indmatch{indmatch}

\opn\aff{aff} \opn\con{conv} \opn\relint{relint} \opn\st{st}
\opn\lk{lk} \opn\cn{cn} \opn\core{core} \opn\vol{vol}
\opn\link{link} \opn\star{star}
\opn\gr{gr}

\def\pot#1#2{#1[\kern-0.28ex[#2]\kern-0.28ex]}

\opn\dirlim{\underrightarrow{\lim}}
\opn\inivlim{\underleftarrow{\lim}}
\let\to=\rightarrow

\def\Implies{\ifmmode\Longrightarrow \else
        \unskip${}\Longrightarrow{}$\ignorespaces\fi}
\def\implies{\ifmmode\Rightarrow \else
        \unskip${}\Rightarrow{}$\ignorespaces\fi}
\def\iff{\ifmmode\Longleftrightarrow \else
        \unskip${}\Longleftrightarrow{}$\ignorespaces\fi}

\let\:=\colon
\newtheorem{Theorem}{Theorem}[section]

\newtheorem{Corollary}[Theorem]{Corollary}
\newtheorem{Proposition}[Theorem]{Proposition}

\let\epsilon\varepsilon
\let\phi=\varphi
\let\kappa=\varkappa
\def\qed{\ifhmode\textqed\fi
      \ifmmode\ifinner\quad\qedsymbol\else\dispqed\fi\fi}
\def\textqed{\unskip\nobreak\penalty50
       \hskip2em\hbox{}\nobreak\hfil\qedsymbol
       \parfillskip=0pt \finalhyphendemerits=0}
\def\dispqed{\rlap{\qquad\qedsymbol}}

\opn\dis{dis}
\def\pnt{{\raise0.5mm\hbox{\large\bf.}}}

\opn\Lex{Lex}

\newcommand{\inD}[1][\relax]{\def\argone{#1}\def\temprelax{\relax}
  \ifx\argone\temprelax\right.\else\,\middle|#1\right.{}\fi}

\newif\ifbinary
\binarytrue
\begin{document}

\title{On the binomial edge ideals of proper interval graphs}

\author{Herolistra Baskoroputro}
%\thanks{}

\address{Herolistra Baskoroputro, Abdus Salam School of Mathematical Sciences (GC University Lahore), 68-B New Muslim Town, Lahore 54600, Pakistan} \email{h.baskoroputro@gmail.com}

%\address{Viviana Ene, Faculty of Mathematics and Computer Science, Ovidius University, Bd.\ Mamaia 124,
% 900527 Constanta, Romania} \email{vivian@univ-ovidius.ro}

\thanks{The author would like to thank the Facuty of Mathematics and Computer Science of the "Ovidius" University of Constanta for the hospitality during the visit of author. The visit of the author to Constanta was supported by ASSMS, which mainly funded by HEC Pakistan}

\begin{abstract}
We prove several cases of the Betti number conjecture for the binomial edge ideal $J_G$ of a proper interval graph  $G$ (also known as closed graph). Namely, we show that this conjecture is true for the linear strand of $J_G$, and true in general for any proper interval graph $G$ such that the regularity of $S/J_G$ equals two.

\end{abstract}
\maketitle

\section*{Introduction}
The proper interval graphs are known since a while in combinatorics. They were first introduced in \cite{H}. A finite simple undirected graph $G$ on the vertex set $[n]$ is called a \emph{proper interval graph} (in brief \emph{PI graph}) if it admits a proper interval ordering. This means that there exists a labeling of the vertices of $G$ such that for any $1 \leq i < j < k \leq n$, if $\{i,k\}$ is an edge of $G$, then $\{i,j\}$ and $\{j,k\}$ are edges of of $G$ as well \cite[Theorem 1]{LO}. PI graphs are also known as unit interval graphs or indifference graphs. Several other properties and characterizations of PI graphs can be found in \cite{R}, \cite{Ro}, \cite{Fi} \cite{Ga}, \cite{Go}, \cite{HMP}.

Binomial edge ideals were introduced in \cite{HHHKR} and \cite{Oh}. They are defined as follows:
If $G$ is a simple graph on $[n]$, then its associated binomial ideal $J_G \subset S=K[x_1,\ldots,x_n,y_1,\ldots,y_n]$ is generated by the binomials $f_{ij}=x_iy_j-x_jy_i, 1\leq i < j\leq n$ with $\{i,j\}\in E(G).$

Various properties of binomial edge ideals have been studied in several papers, and some interesting still open questions in this topic exist. For example, one of the most intriguing conjectures regards binomial edge ideals associated with PI graphs. This conjecture was stated in \cite{EHH} and it claims that, for any PI graph $G$ on the vertex set $[n]$, $J_G$ and its initial ideal with respect to the lexicographic order share the same graded Betti numbers. We shall refer to this conjecture as the \textit{Betti number conjecture for PI graphs}. So far, this conjecture was proved for PI graphs whose binomial edge ideals are Cohen-Macaulay \cite[Propositin 3.2]{EHH}.

%In \cite[Theorem 10]{Sara2}, the authors has proved that for a pair of graph $G_1,G_2$, reg$(S/J_{G_1,G_2})=1$ if and only if
%both of $G_1, G_2$ are complete graphs and one of them is $P_2$. But this actually means that $S/J_G =1$ (i. e., $S/J_G$ has a linear resolution) if and only if $G$ is a complete graph. If $G$ is a complete graph and one order the generators of $J_G$ decreasingly with respect to lexicographical order $<$, then $S/\ini_<(J_G)$ has a linear resolution. The equality $\beta_{ij}(S/J_G) = \beta_{ij}(S/\ini_<(J_G))$ is clear for this case. Therefore, we consider the next non-trivial case.

In Theorem \ref{main2}, we prove this conjecture for any PI graph $G$ with $\reg(S/J_G) =2$.
The first main step in proving Theorem \ref{main2} is Theorem \ref{main} where we show that $S/J_G$ and $S/\ini_<(J_G)$ share the same linear stand in the Betti diagram for any PI graph $G$.
%Afterwards, by using a nice property of $J_G$ when reg$(S/J_G)=2$, we derive the equality \begin{equation} \beta_{ij}(S/J_G) = \beta_{ij}(S/\ini_<(J_G)) \end{equation} for any $i,j$ if reg$(S/J_G)=2$.

The paper is organized as follows. In Section \ref{recall}, we recall basic facts about binomial edge ideals of PI graphs and their initial ideals. In Section \ref{Initial strand}, we prove Theorem \ref{main} which states that, if $G$ is a PI graph, then \begin{equation} \beta_{i,i+1}(S/J_G) = \beta_{i,i+1}(S/\ini_<(J_G)) = i f_i (\Delta(G)),\end{equation} where $f_i (\Delta(G))$ denotes the number of cliques with $i+1$ vertices in the clique complex $\Delta(G)$ of $G$. Finally, we prove the Betti number conjecture for any PI graph $G$ in the case that $\reg(S/J_G)=2$.

\section{Preliminaries}\label{recall}
In this section we  review fundamental results on binomial edge ideals that will be used in the next sections. To begin with, we fix some notation and present the basic notions which we use in the main sections.

Let $K$ be a field. Let $[n]=1,2,\ldots,n$ for $n \in \mathbb{N}.$ Let $G$ be a simple graph on the vertex set $[n]$. This means that $G$ has no loops, no multiple edges, and it is undirected. We denote the edge set of $G$ by $E(G)$. For graphs, we use the standard terminology and notation. For example, if $\mathcal{S} \subset [n]$, $G_\mathcal{S}$ denotes the restriction of $G$ to $\mathcal{S}$ and $G^c$ denotes the complement of the graph $G$.

The associated \textit{binomial edge ideal} of $G$ is $J_G \subset S=K[x_1,\ldots,x_n,y_1,\ldots,y_n]$ which is generated by the binomials $f_{ij}=x_iy_j-x_jy_i, 1\leq i < j\leq n$ with $\{i,j\}\in E(G).$ It is clear that we can neglect isolated vertices of $G$, hence we shall assume that our graphs have no isolated vertex throughout this paper.

In the pioneering paper \cite{HHHKR}, it is shown that the generators of $J_G$ form a (quadratic) Gr\"obner basis with respect to the lexicographic order induced by the natural ordering of the indeterminates if and only if $G$ is a \textit{proper interval (PI) graph}. It should be noted that in \cite{HHHKR}, the authors use the term "closed graph" for the PI graph. Nevertheless, we use the notion PI graph which has been well known in combinatorics since around 60 years \cite{H}. The equivalence of PI graphs and closed graphs is shown in \cite{CR2}. Some other papers that discuss properties of PI graphs related with commutative algebra are \cite{CA}, \cite{Cr}, \cite{CR}, \cite{EHH}, \cite{EHH2}, \cite{EHH3}, \cite{M}.

For a graph $G$, its \textit{clique complex} $\Delta(G)$ is the simplicial complex of all its cliques, that is, all complete subgraphs of $G$. The maximal cliques of $G$ are called \textit{facets} of $\Delta(G)$. In \cite[Theorem 2.2]{EHH}, it is shown that $G$ is a PI graph if and only if there exists a labeling of $G$ such that all the facets of the clique complex $\Delta(G)$ of $G$ are intervals $[a,b] = \{a,a+1,\ldots,b-1,b\}\subset [n].$ This means, in particular, that if $\{i,j\}\in E(G)$, then for any $i\leq  k< \ell \leq j$, $\{k,\ell\}\in E(G).$ When a PI graph $G$ is given, we always assume that its vertices are labeled such that the facets of its clique complex are of the form $ [a_i,b_i], 1 \leq i \leq r,$ with $1 = a_1 < a_2 < \cdots < a_r < b_r = n.$

%Let $G$ be a simple graph on the vertex set $[n]$ and $S=K[x_1,\ldots,x_n,y_1,\ldots,y_n]$ the polynomial ring in $2n$ indeterminates over a field $K$ endowed with the lexicographic order $<$ induced by $x_1>\cdots >x_n>y_1>\cdots >y_n.$ Let $J_G=(f_{ij}: \{i,j\}\in E(G))\subset S$ be the associated binomial edge
%ideal of $G$.  It is obvious that, for computing the depth and regularity of $S/J_G,$ we may assume without loss of generality that $G$ has no isolated vertex.
%Indeed, if $F$ is a facet of $\Delta(G)$ which contains $\{i,j\}$, then $F$ contains all the edges $\{k,\ell\}$ where $i\leq  k< \ell \leq j$.

Let $G$ be a PI graph. Let $<$ be the lexicographical order on $S$, induced by $x_1 > \ldots x_n > y_1 > \ldots >y_n$. Then the initial ideal of the binomial ideal $J_G$ is the monomial ideal, $\ini_<(J_G)=(x_iy_j: \{i,j\}\in E(G))$. This is the monomial edge ideal of a bipartite graph on the vertex set $\{x_1,\ldots,x_n\}\cup\{y_1,\ldots,y_n\}$ with edge set $\{\{x_i,y_j\}:\{i,j\} \in E (G)\}$. We set $\ini_<(G)$ to be this bipartite graph. Therefore, we have $\ini_<(J_G)=I(\ini_<(G)).$

\section{Linear strand of a binomial edge ideal of a PI graph and of its initial ideal}\label{Initial strand}
We recall from \cite{RVT} the formula for the linear strand of an edge ideal of a graph. Let $H$ be a graph on the vertex set $[n]$, and let $I(H) \subset R=K[x_1,\ldots,x_n]$ be its monomial edge ideal.

\begin{Proposition}\cite[Proposition 2.1]{RVT}
\label{linear strand of the edge ideal}
\begin{equation*}\beta_{i,i+1}(R/I(H)) = \sum_{\mathcal{S} \subseteq V_H, |\mathcal{S}| = i + 1} (\# \comp(H^c_\mathcal{S})-1),\end{equation*} where $\#\comp(H^c_\mathcal{S})$ is the number of the connected components of $H^c_\mathcal{S}$. \end{Proposition}

From the above proposition, for bipartite graphs we can derive a more specific formula.

\begin{Corollary}\label{bipartite}
Let $H$ be a bipartite graph on the vertex set $V(H)$. Then,
\[\beta_{i,i+1}(R/I(H)) = \#\{\mathcal{S} \subseteq V(H) : |\mathcal{S}| = i + 1, H^c_\mathcal{S} \text{ has 2 connected components}\}.\]
\end{Corollary}

\begin{proof}
Without lost of generality, we may take $A=\{x_1, \ldots, x_{|A|}\}$ and $B=\{y_1,\ldots, y_{|B|}\}$, the bipartition of $H$. Thus, $V(H)= A\bigcup B$.
Let $\mathcal{S}$ be a subset of $V(H)$. If $\mathcal{S} \subseteq A$ or $\mathcal{S} \subseteq B$, then $H_\mathcal{S}^c$ is a complete graph on vertex set $\mathcal{S}$ which has only one connected component.

The other possible choice is when $\mathcal{S}$ consists of a subset of $A$ and a subset of $B$. That is, $\mathcal{S} = \{x_{k_1}, \ldots, x_{k_a}:1 \leq k_1 < \ldots <k_a \leq |A|\} \bigcup \{y_{\ell_1}, \ldots , y_{\ell_b}:1 \leq l_1 < \ldots < \ell_b \leq |B|\}$.

Consider the case when $H_{\mathcal{S}}$ is a complete bipartite graph, i.e., $\{x_{k_i},y_{\ell_j}\} \in E(H)$ for all $i \in [a], j \in [b].$ Then, $H^c_\mathcal{S}$ has 2 connected components, since it does not have any edge between $A \bigcap \mathcal{S}$ and $B\bigcap \mathcal{S}$.

The last case is when $H_{\mathcal{S}}$ is a non-complete bipartite graph. Then there exist $i \in [a], j \in [b]$ such that $\{x_{k_i},y_{\ell_j}\} \not\in E(H)$. Then $\{x_{k_i},y_{\ell_j}\} \in E(H^c_\mathcal{S})$. Since in $H^c_\mathcal{S}$ the vertex $x_{k_i}$ is connected to all other vertices in $A \bigcap \mathcal{S}$ and the vertex $y_{\ell_j}$ is connected to all other vertices in $B \bigcap \mathcal{S}$, then $H^c_\mathcal{S}$ has only one connected component.

Therefore, for a bipartite graph $H$ and any $\mathcal{S} \subseteq V(H)$, $H^c_\mathcal{S}$ has either one or two connected components, which means that \#comp$((H^c_\mathcal{S})-1)$ is either 0 or 1, respectively. Hence, by Proposition \ref{linear strand of the edge ideal}, we have $\beta_{i,i+1}(R/I(H))$ equals the number of subset $\mathcal{S}$ with $i+1$ vertices such that $H^c_\mathcal{S}$ has two connected components. \end{proof}

The following proposition is a particular case of Corollary 4.3 in \cite{HKM}.
\begin{Proposition}\label{clique strand}
Let $G$ be a graph on the vertex set $[n]$ and $J_G$ be its binomial edge ideal. Let $\Delta(G)$ be the clique complex of $G$ and let $f_i(\Delta(G))$ denote the number of the cliques of $G$ with $i+1$ vertices. Then \begin{equation} \beta_{i,i+1}(S/J_G) = if_i(\Delta(G)).\end{equation}
\end{Proposition}

\begin{proof} cf. \cite[Corollary 4.3]{HKM}. \end{proof}

Now we are ready to prove our first main theorem.

\begin{Theorem}\label{main}
Let $G$ be a PI graph over the vertex set $[n]$. Let $<$ be the lexicographical order on $S$, induced by $x_1 > \cdots > x_n > y_1 > \cdots >y_n$. Then we have \begin{equation} \beta_{i,i+1}(S/J_G) = \beta_{i,i+1}(S/\ini_<(J_G)) = if_i(\Delta(G)).\end{equation}
\end{Theorem}

\begin{proof}
By Proposition \ref{clique strand}, we know that $\beta_{i,i+1}(S/J_G) = if_i(\Delta(G)).$ Hence, we need to show that $\beta_{i, i +1}(S/\ini_<(J_G))= if_i(\Delta(G)).$

As we have stated in the introduction, we may assume that the graph $G$ is labeled such that the facets of $\Delta(G)$ are the intervals $[a_1,b_1],\ldots,[a_r,b_r]$ where $1 =a_1 < a_2 < ... < a_r < b_r = n.$ Since $G$ is a PI graph, one can consider the bipartite graph $H = \ini_<(G)$ on the vertex set $V(H) = \{x_1,\ldots,x_n\} \bigcup \{y_1,\ldots,y_n\}$ with the edge set $E(H)= \{ \{x_r,y_s\}: r < s \text{ and } \{r,s\} \in E(G)\}$. Observe that $I(H)$ is an ideal in $K[x_1,\ldots,x_n,y_1,\ldots,y_n] = S$.

From Corollary \ref{bipartite}, we have \begin{eqnarray*}& \beta_{i,i+1}(S/ \ini_<(J_G)) = \beta_{i,i+1}(S/I(H)) = \\
 & = \#\{\mathcal{S} \subseteq V(H) : |\mathcal{S}| = i + 1, H^c_\mathcal{S} \text{ has 2 connected components}\}.\end{eqnarray*}

Let $X=\{x_1,\ldots,x_n\}$, $Y=\{y_1,\ldots,y_n\}$, and $\mathcal{S} \subset V(H)$ be a set with $|\mathcal{S}|=i+1$.
By the proof of Corollary \ref{bipartite}, we know that $H^c_\mathcal{S} $ has two connected components if and only if $\mathcal{S} \bigcap X$ and $\mathcal{S} \bigcap Y$ are nonempty and $\{x_r,y_s\}\in E(H)$ for all $x_r \in \mathcal{S}\bigcap X$ and $y_s \in \mathcal{S} \bigcap Y$.

Let $\mathcal{S}\bigcap X = \{x_{k_1},\ldots,x_{k_j}\}, k_1 < \cdots < k_j,$ and $\mathcal{S}\bigcap Y = \{y_{k_{j+1}},\ldots,y_{k_{i+1}}\}, k_{j+1} < \cdots < k_{i+1},$ for an integer $1 \leq j \leq i$. Then, $H_\mathcal{S}^c$ has two connected components if and only if $\{x_{k_r},y_{k_s}\} \in E(H)$ for all $1 \leq r \leq j$ and $j +1 \leq s \leq i +1$, which is equivalent to saying that $\{k_r,k_s\} \in E(G)$ for all $1 \leq r \leq j$ and $j+1 \leq s \leq i+1$. This implies that $\{k_1,\ldots,k_j,k_{j+1},\ldots,k_{i +1}\}$ is a clique of $G$.

This shows that, for any $\mathcal{S} \subset V(H)$ such that $|\mathcal{S}| = i +1$ and $H_\mathcal{S}^c$ has two connected components, we may associate it to a pair $(j,C_{i+1})$ where $1 \leq j \leq i$ is an integer and $C_{i+1}$ is a clique of the graph $G$ with $i+1$ vertices.

Conversely, let $1 \leq j \leq i$ be an integer and let $C_{i+1} =\{k_1,\ldots,k_{i+1}\}$ be a clique of $G$, with $k_1 < \cdots < k_{i+1}$. Then, if $\mathcal{S} = \{x_{k_1},\ldots,x_{k_j}\} \bigcup \{y_{k_{j+1}},\ldots,y_{k_{i+1}}\}$, we have $|\mathcal{S}|=i+1$ and $\{k_r,k_s\}\in E(G)$ for all $1 \leq r \leq j$ and $j+1 \leq s \leq i +1$, hence $\{x_{k_r},y_{k_s}\} \in E(H)$. Thus, $H_\mathcal{S}^c$ has two connected components, namely the complete graphs on the vertex sets $\{x_{k_1},\ldots,x_{k_j}\}$ and $\{y_{k_{j+1}},\ldots,y_{k_{i+1}}\}$, respectively. Obviously, the above maps $\mathcal{S} \to (j,C_{i+1})$ and $(j,C_{i+1}) \to \mathcal{S}$ are inverse.

Therefore, we have obtained the following equality:
\begin{eqnarray*} & \#\{\mathcal{S} \subseteq V(H) : |\mathcal{S}| = i + 1, H^c_\mathcal{S} \text{ has 2 connected components}\} = \\
& = \#\{(j,C_{i+1}) : 1 \leq j \leq i \text{ and $C_{i+1}$ is a clique of $G$ with $i+1$ vertices }\},
\end{eqnarray*} which implies that $\beta_{i, i +1}(S/I(H))= if_i(\Delta(G)).$
\end{proof}

The condition that $G$ is a PI graph in Theorem \ref{main} cannot be omitted. For example, consider the graph $G$ on the vertex set $\{1,2,3,4\}$ and with edge set $\{\{1,2\},\{1,3\},\{1,4\}\}$. Then $\beta_{2,3}(S/J_G)=\beta_{3,4}(S/J_G)=0$, while $\beta_{2,3}(S/\ini_<J_G)=3$ and $\beta_{3,4}(S/\ini_<J_G)=1$.

\section{The Betti number theorem for PI graphs associated with binomial edge ideals with small regularity}
Before we discuss our main theorem about the graded Betti numbers for the binomial edge ideals of a PI graph $G$ with reg$(S/J_G)=2$, first we show when one gets $\reg(S/(J_G))=2$.

\begin{Proposition}\label{characterization reg 2}
Let $G$ be a PI graph on the vertex set $[n]$. Then $reg(S/J_G) =2$ if and only if $G$ is in one of the following forms:
\begin{enumerate}
\item $G = K_m \bigcup K_p$ with $m + p =n$,
\item $G$ is connected  and $\Delta(G)$ is generated by two maximal cliques of the form $[1,b],[a,n]$ where $1 < a \leq b <n$.
\end{enumerate}
\end{Proposition}

\begin{proof}
First consider the case when $G$ is not a connected graph, that is, $G = G_1 \bigcup \cdots \bigcup G_c, c \geq 2 $, where $G_1 \bigcup \cdots \bigcup G_c$ are the connected components of $G$. Then \[\reg(S/J_G) = \reg(S/J_1) + \cdots + \reg(S/J_c). \] From \cite[Theorem 3.2]{EZ}, $\reg(S/J_G) = \ell_1 + \cdots + \ell_c$, where $\ell_i$ denotes the length of a longest induced path in $G_i, 1 \leq i \leq c$. If reg$(S/J_G)=2$, the equality $\ell_1 + \cdots + \ell_c =\reg(S/J_G) = 2$ occurs if and only if $G$ has only two connected components and $\ell_1 = \ell_2 =1$, which means that $G_1, G_2$ are complete graphs.

%Now for the other way around, let $G = G_1 \cup G_2$ where $G_1 = K_m, G_2= K_p, m+p = n$. Then, \[S/J_G \cong S_1/J_{G_1} \bigotimes_K S_2/J_{G_2},\] where $S_i = K[\{x_j,y_j\}_{j \in V(G_i)}], i =1,2.$ As $G_1$ and $G_2$ are complete graph, by \cite[Theorem 2.1]{Sara}, $\reg(S_i/J_i) =1$ for $i=1,2$. Hence, $\reg(S/J_G)=2$.

%Let $S_1 = K[x_1, \ldots, x_m]$ and $S_2= K[x_{m +1}, \ldots, x_n]$ be the polynomial rings with variables related to the vertices in $G_1$ and $G_2$, respectively. Then by \cite{BC}, depth$(S_1/J_{G_1})$= depth$(S_1/J_{K_m})=$  dim $(S_1/J_{K_m})= m+1$. Then by Auslander-Buchsbaum formula, pd$(S_1/J_{G_1}) = m-1$. Similarly, we have pd$(S_2/J_{G_2}) = p-1$.
%
%Thus, let \begin{equation*}
%\mathbb{F}_1:0 \to S_1(-m)^{\beta_{m-1,m}(S_1/J_{G_1})} \to \ldots \to S_1(-2)^{\beta_{1,2}(S_1/J_{G_1})} \to S_1 \to S_1/J_{G_1} \to 0\end{equation*}
%be the graded free resolution of $S_1/J_{G_1}$, and
%let \begin{equation*}
%\mathbb{F}_2:0 \to S_2(-p)^{\beta_{p-1,p}(S_2/J_{G_2})} \to \ldots \to S_2(-2)^{\beta_{1,2}(S_2/J_{G_2})} \to S_2 \to S_2/J_{G_2} \to 0\end{equation*}
%be the graded free resolution of $S_2/J_{G_2}$.
%Hence, in the resolution of $S/J_G$, the maximal shift will be $(i+1) + (j+1)$ which occurred in degree $i+j$. This implies reg$(S/J_G)=2$.
The other case is when $G$ is connected. By \cite[Theorem 3.2]{EZ}, the length of the longest induced path in $G$ is equal to the regularity, which is 2. This holds if and only if $\Delta(G)$ has two maximal cliques of the form $[1,b],[a,n]$ where $1 < a \leq b <n$. Indeed, suppose that $\Delta(G)$ has at least $3$ facets, that is, the facets are $[1,b_1],[a_2,b_2],\ldots, [a_r,b_r=n], r \geq 3$. Then there exists an induced path in $G$ of length at least $3$ which contains the vertices $1,a_2,b_2,b_3$.
\end{proof}

%\begin{proof}
%In \cite[Theorem 2.2]{EZ}, \begin{equation} reg(S/J_G) = reg(S/\ini_<(J_G))\end{equation} for a PI graph $J_G$. This means that $reg(S/\ini_<(J_G)) = 2$. Hence, $\beta_{i,j}(S/\ini_<(J_G) = 0$ for all $j > i +1$. This equations combined with Theorem \ref{main} yield our desired result.
%\end{proof}

Now we are ready to prove our main result, which shows that the Betti number conjecture for PI graphs is true in the case where reg$(S/J_G)=2$.
\begin{Theorem}\label{main2}

Let $G$ be a PI graph over vertex the set $[n]$. If reg$(S/J_G) = 2$, then
\begin{equation*} \beta_{ij}(S/J_G) = \beta_{ij}(S/\ini_<(J_G)) \end{equation*} for all $i,j$.
\end{Theorem}

\begin{proof}
Notice that since $G$ is a PI graph, then, from \cite[Theorem 3.2]{EZ} we have reg$(S/\ini_<(J_G)) = \reg(S/J_G) =2$. Therefore, $\beta_{ij}(S/\ini_<(J_G)) = 0 = \beta_{ij}(S/J_G)$ for all $j \geq i + 3$. We also have Theorem \ref{main} for $j = 1+1$. Therefore, we need to prove the equality only for $j = i+2$.

We have the following cases as in Proposition \ref{characterization reg 2}:\\
\textbf{Case 1:} Let $G = G_1 \cup G_2$ where $G_1 = K_m, G_2= K_p, n = m+p$. Let $S_1 = K[x_1, \ldots, x_m, y_1, \ldots, y_m]$ and $S_2= K[x_{m +1}, \ldots, x_n,y_{m+1},\ldots,y_{n}]$ be the polynomial rings with variables related to the vertices in $G_1$ and $G_2$, respectively. As $G_1,G_2$ are complete graphs, it is well known that, since the binomial edge ideal of a complete graph and its initial ideal have a linear resolution, \begin{equation} \beta_{ij}(S_k/(J_{G_k})) = \beta_{ij}(S_k/\ini_<(J_{G_k})), k = 1,2. \end{equation}

The resolution of $S/J_G$ is obtained by tensoring the resolution of $S_1/J_{G_1}$ with the resolution of $S_2/J_{G_2}$, and the resolution of $S/(\ini_<(J_G))$ is obtained by tensoring the resolution of $S_1/(\ini_<J_{G_1})$ with the resolution of $S_2/(\ini_<J_{G_2})$. Hence, the Betti numbers of $S/J_G$ and $S/(\ini_<(J_G))$ are equal.

\textbf{Case 2:} Let $G$ be a connected graph with $\Delta(G) = <[1,b],[a,n]>$, with $1 < a \leq b < n.$ Then, by \cite[Section 3]{HHHKR}, obviously $J_G$ has two minimal primes, namely $P= J_{K_n}$ and $(\{x_i,y_i\}_{a \leq i \leq b},J_{K_{[1,a-1]}},J_{K_{[b+1,n]}})= Q$. We denote by $K_{[a,b]}$ the complete graph on the vertex set $[a,b]$.

This implies, by \cite[Theorem 3.2]{HHHKR}, that $J_G = P \bigcap Q$.
We have the following exact sequence:
\begin{equation}\label{exact jg} 0 \to S/J_G \to S/P \bigoplus S/Q \to S/(P+Q) \to 0. \end{equation}

Observe that \[P + Q = J_{K_n} + (\{x_i,y_i\}_{a \leq i \leq b},J_{K_{[1,a-1]}},J_{K_{[b+1,n]}})=J_{K_n}+(\{x_i,y_i\}_{a \leq i \leq b}) = J_{K_{[n] \backslash [a,b]}}+(\{x_i,y_i\}_{a \leq i \leq b}).\]

To simplify the notation, we set $\Tor_k(M)_\ell := \Tor_k(M,K)_\ell$ for any $k,\ell \in \mathbb{N}$. We consider the following long exact sequence of Tor which follows from (\ref{exact jg}):
\begin{eqnarray}\label{tor jg}
& \cdots \to \Tor_{i+2}(\frac{S}{J_G})_{i+2} \to (\Tor_{i+2}(\frac{S}{P}) \bigoplus \Tor_{i+2}(\frac{S}{Q}))_{i+2} \to \Tor_{i+2}(\frac{S}{P+Q})_{i+2} \to \nonumber \\
& \to \Tor_{i+1}(\frac{S}{J_G})_{i+2} \to (\Tor_{i+1}(\frac{S}{P}) \bigoplus \Tor_{i+1}(\frac{S}{Q}))_{i+2} \to \Tor_{i+1}(\frac{S}{P+Q})_{i+2} \to \nonumber \\
& \to \Tor_{i}(\frac{S}{J_G})_{i+2} \to (\Tor_{i}(\frac{S}{P}) \bigoplus \Tor_{i}(\frac{S}{Q}))_{i+2} \to \Tor_{i}(\frac{S}{P+Q})_{i+2} \to \cdots \end{eqnarray}

Obviously, we have have $\Tor_{i+2}(\frac{S}{J_G})_{i+2}=0$ and $\Tor_{i+2}(\frac{S}{P})_{i+2}=\Tor_{i+2}(\frac{S}{J_{K_n}})_{i+2}=0$.
As \[\reg(\frac{S}{P+Q})
=\reg(\frac{K[\{x_i,y_i\}_{i \in [n] \setminus [a,b]}]}{J_{K_{[n] \setminus [a,b]}}})
+\reg(\frac{K[\{x_i,y_i\}_{a \leq i  \leq b}]} {(x_i,y_i: a \leq i \leq b)})= 1 + 0 =1,\]
we get $\beta_{ij}(\frac{S}{P+Q})=0$ for $j \geq i +2,$ which implies that $\Tor_{i}(\frac{S}{P+Q})_{i+2} =0.$

From (\ref{tor jg}), we derive the following equality:
\begin{eqnarray}\label{betta jg}
&\beta_{i+2,i+2}(\frac{S}{Q}) - \beta_{i+2,i+2}(\frac{S}{P+Q})+\beta_{i+1,i+2}(\frac{S}{J_G})-(\beta_{i+1,i+2}(\frac{S}{P})+\beta_{i+1,i+2}(\frac{S}{Q})) \nonumber \\
& +\beta_{i+1,i+2}(\frac{S}{P+Q})-\beta_{i,i+2}(\frac{S}{J_G})+\beta_{i,i+2}(\frac{S}{Q}) = 0. \end{eqnarray}

Now consider the ideal $\ini_<(J_G)$. From \cite[Lemma 1.3]{C}, we know that $\ini_<(J_G) = \ini_<(P) \bigcap \ini_<(Q)$ if and only if $\ini_<(P) + \ini_<(Q) = \ini_<(P+Q)$. The latter equality is equivalent to \[\ini_<(J_{K_n}) + (\{x_i,y_i\}_{a \leq i \leq b},\ini_<(J_{K_{[1,a-1]}}),\ini_<(J_{K_{[b+1,n]}})) = (\{x_i,y_i\}_{a \leq i \leq b}) + \ini_<(J_{[n]\backslash [a,b]}).\] But this is obviously true.

Hence, we also have the following exact sequence:
\begin{equation}\label{exact injg} 0 \to S/\ini_<(J_G) \to S/\ini_<(P) \bigoplus S/\ini_<(Q) \to S/\ini_<(P+Q) \to 0. \end{equation}

As in the case of $J_G$, we can consider the following exact sequence of Tor for $\ini_<(J_G)$ which follows from (\ref{exact injg}):\\
\begin{eqnarray*}
& \cdots \to \Tor_{i+2}(\frac{S}{\ini_<(J_G)})_{i+2} \to (\Tor_{i+2}(\frac{S}{\ini_<(P)}) \bigoplus \Tor_{i+2}(\frac{S}{\ini_<(Q)}))_{i+2} \to \Tor_{i+2}(\frac{S}{\ini_<(P+Q)})_{i+2} \nonumber \\
& \to \Tor_{i+1}(\frac{S}{\ini_<(J_G)})_{i+2} \to (\Tor_{i+1}(\frac{S}{\ini_<(P)}) \bigoplus \Tor_{i+1}(\frac{S}{\ini_<(Q)}))_{i+2} \to  \Tor_{i+1}(\frac{S}{\ini_<(P+Q)})_{i+2} \to \nonumber \\
& \to \Tor_{i}(\frac{S}{\ini_<(J_G)})_{i+2} \to (\Tor_{i}(\frac{S}{\ini_<(P)}) \bigoplus \Tor_{i}(\frac{S}{\ini_<(Q)}))_{i+2} \to \Tor_{i}(\frac{S}{\ini_<(P+Q)})_{i+2} \to \cdots \end{eqnarray*}

With the same arguments as in the case of $J_G$, we obtained the equality:
\begin{eqnarray}\label{betta injg}
& \beta_{i+2,i+2}(\frac{S}{\ini_<(Q)}) - \beta_{i+2,i+2}(\frac{S}{\ini_<(P+Q)})+ \beta_{i+1,i+2}(\frac{S}{\ini_<(J_G)})-\beta_{i+1,i+2}(\frac{S}{\ini_<(P)}) \nonumber \\
 &-\beta_{i+1,i+2}(\frac{S}{\ini_<(Q)}) +\beta_{i+1,i+2}(\frac{S}{\ini_<(P+Q)})
 -\beta_{i,i+2}(\frac{S}{\ini_<(J_G)})+\beta_{i,i+2}(\frac{S}{\ini_<(Q)}) = 0. \end{eqnarray}

Compare equations (\ref{betta jg}) and (\ref{betta injg}).
We know, from Theorem \ref{main}, that \[\beta_{i+1,i+2}(\frac{S}{J_G})=\beta_{i+1,i+2}(\frac{S}{\ini_<(J_G)}).\]
We also know that \[\beta_{i+1,i+2}(\frac{S}{P}) = \beta_{i+1,i+2}(\frac{S}{\ini_<(P)}),\] since $P=J_{K_n}.$ In addition, we have  \[\beta_{i+2,i+2}(\frac{S}{Q}) =\beta_{i+2,i+2}(\frac{S}{\ini_<(Q)}),\]
\[\beta_{i+1,i+2}(\frac{S}{Q}) = \beta_{i+1,i+2}(\frac{S}{\ini_<(Q)}),\] and
\[\beta_{i,i+2}(\frac{S}{Q})= \beta_{i,i+2}(\frac{S}{\ini_<(Q)}),\] due the particular form of the ideal $Q$.

Finally, we also have the following equalities:
 \[\beta_{i+2,i+2}(\frac{S}{P+Q})=\beta_{i+2,i+2}(\frac{S}{\ini_<(P+Q)})\]
 and \[\beta_{i+1,i+2}(\frac{S}{P+Q})=\beta_{i+1,i+2}(\frac{S}{\ini_<(P+Q)})\] due the particular form of the ideal $P+Q.$  

Therefore, $\beta_{i,i+2}(\frac{S}{J_G})=\beta_{i,i+2}(\frac{S}{\ini_<(J_G)})$.
\end{proof}

{}


\begin{thebibliography}{}

%\bibitem{BCP} D. Bayer, H. Charalambous, S. Popescu, {\em Extremal Betti numbers and
%applications to monomial ideals},  J. Algebra \textbf{221} (1999), 497--512.

%\bibitem{BC} W. Bruns, A. Conca, {\em Gr\"obner bases and determinantal ideals}, In: Commutative Algebra, Singularities and Computer Algebra, J. Herzog and V. Vutulescu, Eds., NTAO Science Series 115, (2003), 9-66
%
%\bibitem{BH} W. Bruns, J. Herzog, \textit{Cohen-Macaulay rings}, Revised Ed., Cambridge University Press, 1998.

\bibitem{C} A. Conca, \textit{Gorenstein ladder determinantal rings}, J. London Math. Soc. (2) {\bf 54} (1996), no. 3, 453--474.

\bibitem{CA} D. A. Cox, A. Erskine, {\em On closed graphs}, (2014) arXiv:1306.5149v2.

\bibitem{Cr} M. Crupi, \emph{Closed orders and closed graphs}, \emph{An. \c Stiin\c t. Univ. "Ovidius" Constan\c ta, Ser. Mat.} (2016)\textbf{24}, pp.159-167

\bibitem{CR} M. Crupi, G. Rinaldo, {\em Binomial edge ideals with quadratic Gr\"obner bases}, Electron. J. Combin., {\bf 18} (2011), no. 1, \# P211.

\bibitem{CR2} M. Crupi, G. Rinaldo, {\em Closed graphs are proper interval graphs}, An. St. Univ. Ovidius Constanta, {\bf 22} (2014), no. 3, 37--44.

\bibitem{EHH} V. Ene, J. Herzog, T. Hibi, {\em Cohen-Macaulay binomial edge ideals}, Nagoya Math. J. {\bf 204} (2011), 57--68.

\bibitem{EHH2} V. Ene, J. Herzog, T. Hibi, {\em Koszul binomial edge ideals}, In "Bridging Algebra, Geometry, and Topology"  Springer Proceedings in Mathematics \& Statistics, Vol. 96,  Ibadula, Denis, Veys, Willem (Eds.) page 125--136.

\bibitem{EHH3} V. Ene, J. Herzog, T. Hibi, {\em Linear flags and Koszul filtration}, Kyoto J. Math. {\bf 55}, no. 3 (2015), 517--530.

%\bibitem{EHHQ} V. Ene, J. Herzog, T. Hibi, A. A. Qureshi, {\em The binomial edge ideal of a pair of graphs}, Nagoya Math. J. {\bf 213} (2014), 105--125.

\bibitem{EZ} V. Ene, A. Zarojanu, {\em On the regularity of binomial edge ideals}, Math. Nachr. {\bf 288}(1)(2013),19--24.

\bibitem{Fi} P. C. Fishburn, \emph{Interval graphs and interval order}, Discrete Math. {\bf 55} (1985), 135--149.

\bibitem{Ga} F. Gardi, \emph{The Roberts characterization of proper and unit interval graphs},
Discrete Math., {\bf 307}(22) (2007), 2906--2908.

\bibitem{Go} M. C. Golumbic, Algorithm Graph Theory and Perfect graphs Academic Press, New York, 1980.

\bibitem{H} G. Hajós, \emph{Uber eine Art von Graphen}, Internat. Math. Nachrichten, {\bf 11}(1957) page 65.

\bibitem{HMP} P. Heggernes, D. Meister, C. Papadopoulos, \emph{A new representation of proper interval graphs with an application to clique-width} DIMAP Workshop on Algorithmic Graph Theory 2009, Electron. Notes Discrete Math., {\bf 32} (2009), 27--34.

%\bibitem{HH10} J. Herzog and T. Hibi, Monomial Ideals, Graduate Texts in Mathematics \textbf{260}, Springer, 2010.

\bibitem{HHHKR} J. Herzog, T. Hibi, F. Hreinsdóttir, T. Kahle, J. Rauh, {\em Binomial edge ideals and conditional independence statements}, Adv. Appl. Math. \textbf{45} (2010), 317--333.

\bibitem{HKM} J. Herzog, D. Kiani, S. Saeedi Madani, {\em The Linear Strand of Determinantal Facet Ideals}, arXiv:1508.07592.

%\bibitem{KM1} D. Kiani and S. Saeedi Madani, {\em Some Cohen-Macaulay and unmixed binomial edge ideals}, Comm. Algebra (2015)
%
%\bibitem{KM2} D. Kiani and S. Saeedi Madani, {\em The regularity of binomial edge ideals of graphs},  arXiv:1310.6126
%
\bibitem{LO} P. J. Looges, S. Olariu, \emph{Optimal greedy algorithms for indifference graphs}, Comput. Math. Appl., \textbf{25} (1993), 15--25.

\bibitem{M} K. Matsuda, {\em Weakly closed graphs and F-purity of binomial edge ideals}, arXiv:1209.4300.

%\bibitem{MM} K. Matsuda, S. Murai, {\em Regularity bounds for binomial edge ideals}, J. Commut. Algebra \textbf{5} (2013), no. 1, 141--149.

\bibitem{Oh} M. Ohtani, {\em Graphs and Ideals generated by some $2$-minors}, Commun. Algebra {\bf 39} (2011), no. 3,  905--917.

%\bibitem{P} I.Peeva, \textit{Graded syzygies}, Algebra and Applications {\bf 14}, Springer, 2011
%
%\bibitem{RR} A. Rauf, G.Rinaldo, {\em Construction of Cohen-Macaulay binomial edge ideals}, arXiv:1205.0475

\bibitem{R} F. S. Roberts, {\em Indifference graphs}, In "Proof Techniques in Graph Theory" (F. Harary, ed.), Academic Press, New York (1969), 139--146.

\bibitem{Ro} F. S. Roberts, Graph Theory and Its Applications to Problems of Society, SIAM Press, Philadelphia, PA (1978).

\bibitem{RVT} M. Roth, A. Van Tuyl, {\em On the linear strand of an edge ideal}, Comm. Algebra, {\bf 35}
(2007), no.3, 821--832

%\bibitem{Sara} S. Saeedi Madani and D. Kiani, {\em Binomial edge ideals of graphs}, Electron. J.
%Combin, {\bf 19} (2012), no. 2, \# P44.

%\bibitem{Sara2} S. Saeedi Madani and D. Kiani, {\em On The Binomial Edge Ideal of a Pair of Graphs}, Electron. J.
%Combin, {\bf 20} (2013), no. 1, \# P48.

%\bibitem{SZ} P. Schenzel and S. Zafar, {\em Algebraic Properties of the binomial edge ideal of complete bipartite graphs},arXiv:1301.0789,  Accepted in An. St. Univ. Ovidius Constanta, Ser. Mat
%
%%\bibitem{V} R. Villarreal, {\em Cohen--Macaulay graphs},  Manuscripta Math. {\bf 66} (1990), 277--293.
%
%\bibitem{W} R. Woodroofe, {\em Matching, coverings, and Castelnuovo-Mumford regularity}, Preprint, (2011),
%arXiv:1009.2756v2
%
%\bibitem{ZZ} Z. Zahid, S. Zafar, {\em On the Betti numbers of some classes of binomial edge ideals}, The Electronic
%Journal of Combinatorics. {\bf 20}(2013) no.4 , \# P37.

\end{thebibliography}
\end{document}